\documentclass{amsart}

\usepackage{amssymb, amsmath, amsthm}
\usepackage{mathrsfs}
\usepackage[fleqn,tbtags]{mathtools}
\usepackage[shortlabels]{enumitem}

\usepackage[colorlinks,linkcolor={blue},citecolor={blue},urlcolor={red}]{hyperref}

\let\le\leqslant
\let\leq\leqslant
\let\ge\geqslant
\let\geq\geqslant

\theoremstyle{plain}
\newtheorem{theorem}{Theorem}[section]
\theoremstyle{remark}
\newtheorem{remark}[theorem]{Remark}
\theoremstyle{plain}
\newtheorem{proposition}[theorem]{Proposition}

\newtheorem{lemma}[theorem]{Lemma}
\numberwithin{equation}{section}
\newcommand{\D}{\{-1,1\}}
\newcommand{\E}{\mathbb E}
\newcommand{\cR}{\mathcal R}
\newcommand{\Rad}{\operatorname{Rad}}

\newcommand{\Dom}{\mathsf{D}}
\newcommand{\calL}{\mathscr L}
\newcommand{\dd}{\,\mathrm d}
\newcommand{\R}{\mathbb R}

\title[Dimension-free $H^\infty$-calculus of angle $<\pi/2$]{Dimension-free $H^\infty$-calculus of angle $<\pi/2$ for UMD-valued Ornstein--Uhlenbeck operators}

\begin{document}

\author{Jan van Neerven}

\address{Delft Institute of Applied Mathematics\\
Delft University of Technology\\P.O. Box 5031\\2600 GA Delft\\The Netherlands}
\email{J.M.A.M.vanNeerven@tudelft.nl}

\begin{abstract}
Let $X$ be a UMD Banach space, let $1<p<\infty$, and let $L_d$ be the generator of the Ornstein--Uhlenbeck semigroup $(P_d(t))_{t\geq 0}$ on $L^p(\mathbb R^d,\gamma_d;X)$. We prove that the operators $-L_d$ are $R$-sectorial with a common angle strictly smaller than $\pi/2$ and with bounds independent of $d$. Combining this with the Hieber--Pr\"uss transference theorem and the Kalton--Weis angle comparison, we deduce that the operators $-L_d$ admit bounded $H^\infty$-calculi of a common angle strictly smaller than $\pi/2$, again with dimension-free bounds. The corresponding Walsh $R$-analyticity estimates are proved first and transferred to the Ornstein--Uhlenbeck setting by a central limit argument.
\end{abstract}



\date{\today}

\maketitle

\section{Introduction and main results}
\label{sec:introduction}

The bounded $H^\infty$-functional calculus of generators of symmetric contraction semigroups is well understood in the scalar-valued setting. In particular, if $-A$ is the generator of such a semigroup $(e^{-tA})_{t\geq 0}$ on a scalar $L^p$-space, then Carbonaro and Dragi\v{c}evi\'c~\cite{CarbonaroDragicevic} proved that
\begin{equation*}
 \omega_{H^\infty}(A) \leq \phi_p^*, \qquad \hbox{where }\, \phi_p^*
 = \arcsin\Bigl|1-\frac2p\Bigr|
 < \frac{\pi}{2}, \qquad 1<p<\infty,
\end{equation*}
and the angle $\phi_p^*$ is optimal in this generality. The strict inequality $\omega_{H^\infty}(A) < \pi/2$ is important: an $H^\infty$-calculus of angle strictly smaller than $\pi/2$ is closely tied to $R$-sectoriality and, on UMD spaces, to maximal $L^p$-regularity and stochastic maximal $L^p$-regularity.

The corresponding question for vector-valued $L^p$-spaces is much less settled. Let $X$ be a UMD Banach space and let $(T(t))_{t\geq 0}$ be a $C_0$-contraction semigroup of positive operators on a scalar $L^p$-space. The Hieber--Pr\"uss transference theorem gives a bounded $H^\infty(\Sigma_\nu)$-calculus for the negative generator of the tensor extension of $T$ to $L^p(X)$ for every $\nu\in(\pi/2,\pi)$; see \cite{HieberPruss,HNVW2}. This result is very general, but by itself it does not give a calculus of angle strictly smaller than $\pi/2$.

Let $(P_d(t))_{t\geq 0}$ denote the Ornstein--Uhlenbeck semigroup on $L^p(\mathbb R^d,\gamma_d;X)$, let $L_d$ be its generator, and put
$$
 A_d:=-L_d.
$$
The main result of this paper is that the operators $A_d$ admit a bounded $H^\infty$-calculus of an angle strictly smaller than $\pi/2$, with dimension-free estimates:

\begin{theorem}[Dimension-free $H^\infty$-calculus of angle $<\pi/2$]
\label{thm:main-Hinfty}
Let $X$ be a UMD space and let $1<p<\infty$. There exist and angle $\sigma_{p,X}\in (0,\pi/2)$ and a constant $C_{p,X}\geq 0$, both depending only on $p$ and $X$, such that for every $d\geq 1$ the operator $A_d$ has a bounded $H^\infty(\Sigma_{\sigma_{p,X}})$-functional calculus on $L^p(\mathbb R^d,\gamma_d;X)$ with bound at most $C_{p,X}$, i.e.,
$$
 \| f(A_d) \| \le C_{p,X}\|f\|_\infty
$$
for all $f\in H^1(\Sigma_{\sigma_{p,x}})\cap H^\infty(\Sigma_{\sigma_{p,x}})$.
\end{theorem}

Strict-angle results (asserting an angle strictly less than $\pi/2$) for the $H^\infty$-calculus were previously known under an interpolation assumption on the target space. Hyt\"onen \cite[Corollary~9.10]{HytonenLPS} proved that if $X$ is isomorphic to a closed subspace of $[H,Y]_\theta$, where $H$ is a Hilbert space and $Y$ is UMD, then the negative generator of any symmetric diffusion semigroup on $L^p(\mu;X)$ has a bounded $H^\infty$-calculus for every angle larger than
$$
 (1-\theta)\pi\Bigl|\frac1p-\frac12\Bigr| + \theta\frac{\pi}{2}.
$$
Under the same assumptions on $X$, Xu \cite[Theorem~4 and its proof]{Xu} subsequently obtained a strict-angle calculus for analytic semigroups of regular contractions. In particular, these results apply to UMD Banach lattices $X$. The result proved here is specific to the Ornstein--Uhlenbeck family, but removes the interpolation assumption: $X$ may be an arbitrary UMD space. Related preservation-of-analyticity results for tensorised semigroups were obtained by Arhancet \cite{ArhancetNC, ArhancetPisier, ArhancetBochner}; see also the dilation results in \cite{ArhancetDilation, AFLM}. These works cover many important classes of Banach spaces, but not arbitrary UMD spaces $X$ in the dimension-free form considered here.

The proof of Theorem~\ref{thm:main-Hinfty} starts by proving dimension-free $R$-bounds for the semigroup $(e^{-t\Delta_N})_{t\geq 0}$ and for its derivative family $(t\Delta_Ne^{-t\Delta_N})_{t>0}$. These two estimates combined give dimension-free $R$-analyticity of
the Walsh semigroup, and hence dimension-free strict-angle $R$-sectoriality of $\Delta_N$. The two real-time $R$-bounds are then transferred to the Ornstein--Uhlenbeck setting by a central limit argument, which allows the passage from suitably scaled Walsh functions to Hermite polynomials, and the same $R$-analyticity argument can be applied subsequently. Theorem~\ref{thm:main-Hinfty} is obtained by passing to the infinite product, applying the Hieber--Pr\"uss theorem to obtain a bounded $H^\infty(\Sigma_\nu)$-calculus for every $\nu>\pi/2$, and using the Kalton--Weis theorem on equality of the optimal $H^\infty$- and $R$-sectorial angles to reduce the angle below $\pi/2$.

We conclude the paper by giving a short proof of the result, due to Betancor, Castro, Curbelo and Rodr\'iguez-Mesa \cite{BetancorEtAl}, that the UMD property of $X$ is necessary for the boundedness of the $H^\infty$-calculus of $A_d$ on $L^p(\mathbb R^d,\gamma_d;X)$, for any fixed dimension $d\geq 1$ and exponent $1<p<\infty$.

Related questions concerning the underlying Banach space property and the optimal angle are considered in the companion papers \cite{Neerven-Stein,Neerven-Rsectorial}. In \cite{Neerven-Stein} we identify and study the Banach space properties needed for the dimension-free $R$-sectoriality of vector-valued Ornstein--Uhlenbeck operators, and in \cite{Neerven-Rsectorial} we study the optimal angle problem.

\section{Preliminaries}
\label{sec:preliminaries}

In this section we collect the notions from Banach space theory and operator theory used below. Terminology and notation are standard and follow \cite{HNVW1,HNVW2}. These works cover the various Banach space properties that play a role in this paper, among them the UMD property, $K$-convexity, type and cotype, and the triangular contraction property, as well as the theory of $R$-bounded families of operators, sectorial operators, and the $H^\infty$-calculus.

As in \cite{HNVW1,HNVW2} (cf. \cite[Section~3.2.b]{HNVW1} and \cite[Proposition~6.1.9]{HNVW2}), the meaning of a \emph{Rademacher sequence} depends on the scalar field: over the real scalars the random variables take the values $-1$ and $1$ with equal probability; over the complex scalars they are taken to be uniformly distributed on the unit circle in the complex plane. Thus, in the complex setting used here, a Rademacher sequence is what is also commonly called a \emph{Steinhaus sequence}. We use $\varepsilon_j$ for Rademacher variables defined in this sense.

\subsection{Rademacher spaces, \texorpdfstring{$K$}{K}-convexity, type, and the UMD property}

Let $(\varepsilon_j)_{j\ge 1}$ be a Rademacher sequence, i.e., a sequence of independent Rademacher variables, on a probability space $(\Omega_\varepsilon, \mathbb P_\varepsilon)$.
For a Banach space $X$, we denote by $\Rad(X)$ the closure in $L^2(\Omega_\varepsilon;X)$ of the finite Rademacher sums
$$
 \sum_{j=1}^n\varepsilon_jx_j, \qquad x_1,\ldots,x_n\in X.
$$
Thus
$$
 \Bigl\|\sum_{j=1}^n\varepsilon_jx_j\Bigr\|_{\Rad(X)}
 = \Bigl( \mathbb E_\varepsilon \Bigl\|\sum_{j=1}^n\varepsilon_jx_j\Bigr\|^2 \Bigr)^{1/2}.
$$
We write $\Rad_m(X)$ for the subspace consisting of sums involving only $\varepsilon_1,\ldots,\varepsilon_m$. By the Kahane--Khintchine inequalities, the exponent $2$ in this definition may be replaced by any exponent $q\in(0,\infty)$, at the expense of an equivalent norm with equivalence constants depending only on $q$.

Let $1\leq r\leq 2$. The space $X$ is said to have \emph{type $r$} if there exists a constant $\tau_{r,X}\ge 0$ such that, for all finite sequences $x_1,\ldots,x_n\in X$,
$$
 \Bigl( \mathbb E_\varepsilon \Bigl\|\sum_{j=1}^n\varepsilon_jx_j\Bigr\|^2\Bigr)^{1/2}
  \leq \tau_{r,X} \Bigl(\sum_{j=1}^n\|x_j\|^r\Bigr)^{1/r}.
 $$
The least admissible constant is denoted by $\tau_{r,X}$. Every Banach space has type $1$, Hilbert spaces have type $2$, and $L^p$-spaces have type $\min\{p,2\}$. We say that $X$ has \emph{non-trivial type} if it has type $r$ for some $1<r\le 2$.

A Banach space $X$ is called \emph{$K$-convex} if the operators
$$
 f\mapsto \sum_{j=1}^n \varepsilon_j\mathbb E_\varepsilon(\overline{\varepsilon_j}f),
 \qquad n\geq 1,
$$
are uniformly bounded on $L^2(\Omega_\varepsilon;X)$. By a theorem of Pisier \cite{Pisier}, $K$-convexity is equivalent to having non-trivial type. From this, in turn, one derives that if $X$ is $K$-convex, so are the spaces $L^p(S;X)$ and $\Rad_m(L^p(S;X))$ for $1<p<\infty$, with constants independent of $S$ and $m$.

A Banach space $X$ is a \emph{UMD space} if for some $1<p<\infty$ (equivalently, for all $1<p<\infty$) there exists a constant $\beta_{p,X}\ge 0$ such that for every finite $X$-valued martingale difference sequence $(d_j)_{j=1}^n$ in $L^p$ and every choice of complex scalars $\epsilon_j$ satisfying $|\epsilon_j|=1$, $j=1,\dots,n$,
$$
 \Bigl\|\sum_{j=1}^n\epsilon_jd_j\Bigr\|_{L^p(X)}
 \leq \beta_{p,X} \Bigl\|\sum_{j=1}^nd_j\Bigr\|_{L^p(X)}.
$$
The least admissible constant is denoted by $\beta_{p,X}$ and is called the UMD constant of $X$ at exponent $p$. Every UMD space is $K$-convex, hence by Pisier's theorem it has non-trivial type $r_X>1$. Fix $1<r<\min\{p,r_X\}$. By \cite[Proposition 7.1.4]{HNVW2} and its proof, the spaces $L^p(S;X)$ have type $r$, with a type constant depending only on $p$, $r$, and $X$, and not on the underlying measure space $S$. It follows that
$$
 L^2(\Omega_\varepsilon;L^p(S;X))
$$
has type $r$ with the same dependencies of the constant. Hence the subspaces
$$
 \Rad_m(L^p(S;X)),\qquad m\geq 1,
$$
have type $r$ with constants independent of $m$ and $S$. If $X$ is a UMD space, then so are the spaces $L^p(S;X)$, $1<p<\infty$, and $\Rad_m(X)$, $m\geq 1$.

\subsection{\texorpdfstring{$R$}{R}-bounded families of operators}

Let $E$ and $F$ be Banach spaces. A family $\mathscr T\subseteq\calL(E,F)$ is called \emph{$R$-bounded} if there exists a constant $C\ge 0$ such that
$$
 \Bigl( \mathbb E_\varepsilon \Bigl\| \sum_{j=1}^n\varepsilon_jT_jx_j \Bigr\|_F^2 \Bigr)^{1/2}
 \leq
 C \Bigl( \mathbb E_\varepsilon \Bigl\| \sum_{j=1}^n\varepsilon_jx_j \Bigr\|_E^2 \Bigr)^{1/2}
$$
for all finite choices $T_1,\ldots,T_n\in\mathscr T$ and $x_1,\ldots,x_n\in E$. The least admissible constant is denoted by $\cR_{E,F}(\mathscr T)$, or simply by $\cR_E(\mathscr T)$ when $E=F$. By the Kahane--Khintchine inequalities, the exponent $2$ in this definition may be replaced by any exponent $q\in(0,\infty)$; this gives an equivalent definition, with comparison constants depending only on $q$.

In the language of Rademacher spaces, $R$-boundedness says that for all finite choices $T_1,\ldots,T_m\in\mathscr T$ the diagonal operator
$$
 \sum_{j=1}^m\varepsilon_jx_j \mapsto \sum_{j=1}^m\varepsilon_jT_jx_j
$$
is bounded from $\Rad_m(E)$ to $\Rad_m(F)$, uniformly in $m$ and in the choice of the operators $T_j\in\mathscr T$. Every $R$-bounded family is uniformly bounded. On Hilbert spaces the converse holds as well, and in fact the validity of this converse characterises Hilbert spaces isomorphically within the class of Banach spaces.

A consequence of the UMD property that will be used repeatedly is Stein's inequality for conditional expectations; see \cite[Section~4.2.d]{HNVW1}. Let $X$ be a UMD space, let $(\Omega,\mathcal F,\mu)$ be a measure space, let $(\mathcal F_n)_{n=1}^m$ be a finite filtration, and let $f_1,\ldots,f_m\in L^p(\Omega;X)$ be given functions. Then Stein's inequality reads
\begin{equation}
\label{eq:Stein-inequality}
 \Bigl\| \sum_{n=1}^m \varepsilon_n\mathbb E(f_n | \mathcal F_n)
 \Bigr\|_{L^p(\Omega\times\Omega_\varepsilon;X)}
 \leq \beta_{p,X}
 \Bigl\| \sum_{n=1}^m\varepsilon_nf_n \Bigr\|_{L^p(\Omega\times\Omega_\varepsilon;X)},
\end{equation}
where, as before, $(\varepsilon_n)_{n=1}^m$ is a Rademacher sequence on an auxiliary probability space $(\Omega_\varepsilon, \mathbb P_{\varepsilon})$. In particular, the constant is independent of $m$.

By reversing the indexing, the same estimate applies to a finite decreasing filtration. Thus, the conditional expectations associated with any finite nested family of $\sigma$-algebras form an $R$-bounded family on $L^p(\Omega;X)$, with $R$-bound depending only on $p$ and $X$.

\subsection{Sectorial and \texorpdfstring{$R$}{R}-sectorial operators}

For $0<\theta<\pi$ we consider the open sector
$$
 \Sigma_\theta := \{z\in\mathbb C\setminus\{0\}:\,|\arg z|<\theta\},
$$
with arguments taken in $(-\pi,\pi)$.
Let $A$ be a closed linear operator with domain $\mathsf D(A)$ on a Banach space $E$. Such an operator is called \emph{sectorial of angle} $\omega\in[0,\pi)$ if
$$
 \sigma(A)\subseteq\overline{\Sigma_\omega}
$$
and, for every $\theta\in(\omega,\pi)$,
$$
 \sup_{\lambda\notin\overline{\Sigma_\theta}} \|\lambda(\lambda-A)^{-1}\|<\infty.
$$
The infimum of all admissible angles is denoted by $\omega(A)$ and is called the \emph{sectorial angle} of $A$.

If $A$ is densely defined and sectorial with $\omega(A)<\pi/2$, then $-A$ generates a bounded analytic $C_0$-semigroup; conversely, if $-A$ generates a bounded analytic $C_0$-semigroup, then $A$ is densely defined and sectorial of angle strictly smaller than $\pi/2$.

The operator $A$ is called \emph{$R$-sectorial of angle} $\omega$ if, in addition, for every $\theta\in(\omega,\pi)$ the family
$$
 \bigl\{ \lambda(\lambda-A)^{-1}:\, \lambda\notin\overline{\Sigma_\theta} \bigr\}
$$
is $R$-bounded. The infimum of all angles such that this additional property holds is denoted by $\omega_R(A)$ and is called the {\em angle of $R$-sectoriality} of $A$. Since $R$-boundedness implies uniform boundedness, it is clear that
$$
 \omega(A)\leq \omega_R(A).
$$
For the semigroups considered in this paper, $R$-sectoriality of angle strictly smaller than $\pi/2$ is equivalently expressed by the existence of an $R$-bounded analytic extension to a non-zero sector; we refer to this property as \emph{$R$-analyticity} of the semigroup.

\subsection{The bounded \texorpdfstring{$H^\infty$}{H-infinity}
functional calculus}

Let $0<\sigma<\pi$. Following \cite[Chapter 10]{HNVW2}, we denote by $H^1(\Sigma_\sigma)$ the space of all holomorphic functions $f:\Sigma_\sigma\to\mathbb C$ for which
$$
 \| f\|_{H^1(\Sigma_\sigma)}
 := \sup_{|\nu|<\sigma} \int_0^\infty |f(e^{i\nu}t)|\,\frac{\dd t}{t}  <\infty.
$$
As usual, $H^\infty(\Sigma_\sigma)$ denotes the Banach algebra of all bounded holomorphic functions on $\Sigma_\sigma$, equipped with the supremum norm.

Let $A$ be a sectorial operator on $E$, with angle of sectoriality
$\omega(A)$. Fix $\omega(A)<\sigma<\pi$. For $f\in H^1(\Sigma_\sigma)$ one defines
$$
 f(A) := \frac{1}{2\pi i} \int_{\partial\Sigma_\nu} f(z)(z-A)^{-1}\dd z,
 \qquad \omega(A)<\nu<\sigma.
$$
The integral converges absolutely in $\mathscr L(E)$ and is independent of the choice of $\nu\in (\omega(A),\sigma)$. Indeed, the sectorial resolvent estimate gives
$$
 \|f(A)\| \lesssim
 \int_{\partial\Sigma_\nu} |f(z)|\,\frac{|\dd z|}{|z|} \lesssim \|f\|_{H^1(\Sigma_\sigma)},
$$
and independence of $\nu$ follows from Cauchy's theorem. The mapping $f\mapsto f(A)$ is called the \emph{Dunford calculus} associated with $A$.

The operator $A$ is said to have a \emph{bounded $H^\infty(\Sigma_\sigma)$-calculus} if there exists a constant $C_\sigma\ge 0$ such that
$$
 \|f(A)\| \leq C_\sigma\|f\|_{H^\infty(\Sigma_\sigma)},
 \qquad f\in H^1(\Sigma_\sigma)\cap H^\infty(\Sigma_\sigma).
$$
In this case, the optimal angle is defined as
$$
 \omega_{H^\infty}(A) :=
 \inf\bigl\{ \sigma>\omega(A):\, A\text{ has a bounded }H^\infty(\Sigma_\sigma)\text{-calculus} \bigr\},
$$
and is called the {\em $H^\infty$-angle} of $A$.
In particular,
$$
 \omega(A)\leq \omega_{H^\infty}(A).
$$

In the proof of Theorem~\ref{thm:main-Hinfty} we use the following result of Kalton and Weis \cite{KaltonWeis} in the form as stated in \cite{HNVW2}: if $A$ is a densely defined sectorial operator on a UMD space with a bounded $H^\infty$-calculus, then $A$ is $R$-sectorial and
$$
 \omega_R(A) = \omega_{H^\infty}(A).
$$
Thus, once a bounded $H^\infty$-calculus (with a possibly bad angle) has been obtained (e.g., by the Hieber--Pr\"uss transference theorem), a strict estimate $\omega_R(A)<{\pi}/{2}$ immediately gives $\omega_{H^\infty}(A)<{\pi}/{2}$.

\section{Lifting the Walsh semigroups}
\label{sec:coupling}

We begin by proving a useful representation for lifts of the Walsh semigroups to Rademacher spaces. The arguments in this section and the next two may be viewed as a Rademacher lift of Pisier's technique \cite{Pisier} for proving dimension-free analyticity of the Walsh semigroups.

Throughout Sections~\ref{sec:coupling}--\ref{sec:Walsh-R-analyticity},  $X$ is a UMD space and we assume $1<p<\infty$.
For $N\geq 1$ and $A\subseteq\{1,\ldots,N\}$, put
$$
 w_A(x):=\prod_{i\in A}x_i, \qquad x\in\D^N,
$$
with $w_\varnothing={\bf 1}$, the constant-one function.
We will always assume that $\D^N$ carries the uniform product probability measure. The {\em Walsh number operator} $\Delta_N$ and its associated semigroup $T_N$ are defined by
$$
 \Delta_Nw_A=|A|w_A, \qquad T_N(t):=e^{-t\Delta_N},  \qquad t\geq 0.
$$
Lifting these objects to the Rademacher spaces
$$
 Y_m^N:=\Rad_m(L^p(\D^N;X))
$$
turns an $R$-boundedness estimate into an ordinary estimate for a diagonal semigroup.
The identity \eqref{eq:diagonal-convex-combination} below represents the lifted diagonal semigroup as an average of commuting projections. Stein's inequality controls their products, while the non-trivial type of $Y_m^N$ enters through Pisier's lemma.

Fixing $N\ge 1$, for subsets $S\subseteq\{1,\ldots,N\}$ let $\mathcal F_S:=\sigma(x_i:\,i\in S)$ be the sub-$\sigma$-algebra of $\D^N$ generated by the coordinate maps $x_i:\D^N\to\{-1,1\}$ with $i\in S$ and consider the conditional expectation
$$
 M_S := \mathbb E(\,\cdot\,\vert\mathcal F_S).
$$
Let $u_1,\ldots,u_N$ be independent random variables, uniformly distributed on $[0,1]$, and set
$$
 S_\rho(u):=\{i:\,u_i\leq \rho\},
$$
where $0\leq \rho\leq 1$. Denoting by $\E_u$ the expectation with respect to $u=(u_1,\ldots,u_N)$, we claim that
\begin{equation}
\label{eq:restriction}
 T_N(t)=\E_u M_{S_{e^{-t}}(u)}, \qquad t\ge 0.
\end{equation}
To verify \eqref{eq:restriction}, note that for a Walsh function $w_A$ one has
$$
 M_Sw_A=
 \begin{cases}
  w_A, & A\subseteq S,\\
  0,   & A\not\subseteq S.
 \end{cases}
$$
In fact, if $A\not\subseteq S$, then at least one of the factors $x_i$, $i\in A$, is averaged out by the conditional expectation and has mean zero. Consequently,
$$
 \E_u M_{S_\rho(u)}w_A = \mathbb P\bigl(A\subseteq S_\rho(u)\bigr)w_A.
$$
By independence of the variables $u_i$,
$$
 \mathbb P\bigl(A\subseteq S_\rho(u)\bigr) = \mathbb P(u_i\leq \rho\text{ for all }i\in A) = \rho^{|A|}.
$$
Taking $\rho=e^{-t}$ gives
$$
 \E_u M_{S_{e^{-t}}(u)}w_A = e^{-t|A|}w_A = T_N(t)w_A.
$$
Since the Walsh functions $w_A$ span all functions on $\D^N$, this proves
\eqref{eq:restriction}.

Fix a finite family of times $0<t_1\leq t_2\leq \cdots\leq t_m$ and an integer $m\geq 1$, and define
$$
 \mathsf S(s)\Big(\sum_{j=1}^m\varepsilon_jf_j\Big) := \sum_{j=1}^m\varepsilon_jT_N(st_j)f_j, \qquad s\ge 0.
$$
This defines a $C_0$-semigroup $\mathsf S = (\mathsf S(s))_{s\ge 0}$ on the Rademacher space $Y^N_m$. Its generator is the operator $-\mathsf A$ given by
\begin{equation}
\label{eq:diagonal-generator}
 \mathsf A\Big(\sum_j\varepsilon_jf_j\Big) =\sum_j\varepsilon_jt_j\Delta_Nf_j.
\end{equation}
Note that $\mathsf S$ and its generator depend on the times $t_1, \cdots, t_m$, although this is not expressed in the notation.

For fixed $s>0$ and $u=(u_1,\ldots,u_N)\in[0,1]^N$, set
\begin{equation}
\label{eq:Pu}
 \mathsf P_u(s)\Big(\sum_j\varepsilon_jf_j\Big) :=\sum_j\varepsilon_jM_{S_{e^{-st_j}}(u)}f_j.
\end{equation}
In view of $0<t_1\leq t_2\leq \cdots\leq t_m$ we have
$$
 S_{e^{-st_1}}(u)\supseteq S_{e^{-st_2}}(u) \supseteq\cdots\supseteq S_{e^{-st_m}}(u).
$$
Define the $\sigma$-algebras
$$
 \mathcal F_j^{u,s} := \mathcal F_{S_{e^{-st_j}}(u)} = \sigma\bigl(x_i:\,i\in S_{e^{-st_j}}(u)\bigr), \qquad j=1,\ldots,m.
$$
They satisfy
$$
 \mathcal F_1^{u,s}\supseteq\cdots\supseteq \mathcal F_m^{u,s}.
$$
Hence $(\mathcal F_j^{u,s})_{j=1}^m$ is a decreasing filtration on $\D^N$, and by definition we have
$$
 M_{S_{e^{-st_j}}(u)} = \mathbb E(\,\cdot\,\vert\mathcal F_j^{u,s}).
$$
We may therefore apply Stein's inequality \eqref{eq:Stein-inequality} (and use the Kahane--Khintchine inequalities to switch to $p$-th Rademacher moments), to the $X$-valued functions $f_1,\ldots,f_m$ on $\D^N$ and the decreasing filtration $(\mathcal F_j^{u,s})_{j=1}^m$. Since the constant in Stein's inequality is independent of the length $m$ of the filtration and the Kahane--Khintchine constants are independent of $m$, this gives
\begin{equation}
\label{eq:Stein-Pu}
 \|\mathsf P_u(s)\|_{\calL(Y^N_m)}  \leq K_{p,X},
\end{equation}
where $K_{p,X}$ is a constant independent of $m$, $N$, $u$, $s$, and the times $t_1,\ldots,t_m$. The same bound holds for arbitrary products of the projections. Indeed, the coordinate conditional expectations satisfy
\begin{align}\label{eq:coordinate-conditional-expectations}
 M_SM_R=M_{S\cap R}, \qquad S,R\subseteq\{1,\ldots,N\}.
\end{align}
Indeed, conditioning successively on the coordinates in $R$ and then on those in $S$ leaves precisely the coordinates belonging to both sets. In particular, these conditional expectations commute. It follows from \eqref{eq:coordinate-conditional-expectations} that, for every fixed $s>0$, the family $\{\mathsf P_u(s):\,u\in[0,1]^N\}$ consists of commuting projections.

Now fix $s>0$ and let $u^{(1)},\ldots,u^{(q)}\in[0,1]^N$ be arbitrary. Since each $\mathsf P_{u^{(\ell)}}(s)$ acts diagonally on the Rademacher coordinates, repeated application of \eqref{eq:Pu} gives
\begin{align*}
 \mathsf P_{u^{(1)}}(s)\cdots\mathsf P_{u^{(q)}}(s)
 \Bigl(\sum_j\varepsilon_jf_j\Bigr) =\sum_j\varepsilon_jM_{R_j}f_j,
\end{align*}
where $ R_j := \bigcap_{\ell=1}^q S_{e^{-st_j}}(u^{(\ell)})$.

For every fixed $1\le \ell\le q$, the sets $S_{e^{-st_j}}(u^{(\ell)})$ decrease with $j$. Hence their intersections satisfy $R_1\supseteq R_2\supseteq\cdots\supseteq R_m$, and putting
$$
 \mathcal G_j:=\mathcal F_{R_j} =\sigma(x_i:\,i\in R_j),  \qquad j=1,\ldots,m,
$$
we have $\mathcal G_1\supseteq\cdots\supseteq\mathcal G_m$, and $M_{R_j}=\mathbb E(\,\cdot\,\vert\mathcal G_j)$. Thus Stein's inequality \eqref{eq:Stein-inequality}, again in its decreasing-filtration form, gives
\begin{equation}
\label{eq:product-Pu}
\sup_{q\geq 1}\sup_{u^{(1)},\ldots,u^{(q)}} \|\mathsf P_{u^{(1)}}(s)\cdots\mathsf P_{u^{(q)}}(s)\| \leq K_{p,X}.
\end{equation}

Averaging \eqref{eq:Pu} with respect to $u$ and using \eqref{eq:restriction} in each Rademacher coordinate, we obtain, for $\sum_{j=1}^m\varepsilon_jf_j\in Y^N_m$,
\begin{align*}
 \E_u\mathsf P_u(s) \Bigl(\sum_{j=1}^m\varepsilon_jf_j\Bigr)
 = \sum_{j=1}^m\varepsilon_j \E_u M_{S_{e^{-st_j}}(u)}f_j
 = \sum_{j=1}^m\varepsilon_jT_N(st_j)f_j= \mathsf S(s) \Bigl(\sum_{j=1}^m\varepsilon_jf_j\Bigr).
\end{align*}
Hence
\begin{equation}\label{eq:diagonal-convex-combination}
 \mathsf S(s)=\E_u\mathsf P_u(s).
\end{equation}

For fixed $s$ and $t_1,\ldots,t_m$, the operator $\mathsf P_u(s)$ takes only finitely many values as $u$ ranges over $[0,1]^N$. Indeed, putting $\rho_j=e^{-st_j}$, we have $\rho_1\geq \cdots\geq \rho_m$, and for each coordinate $i\in\{1,\ldots,N\}$ the membership pattern
$$
 \bigl({\bf 1}_{\{i\in S_{\rho_j}(u)\}}\bigr)_{j=1}^m
 = \bigl({\bf 1}_{\{u_i\leq \rho_j\}}\bigr)_{j=1}^m
$$
is of the form $(1,\ldots,1,0,\ldots,0)$. Thus there are at most $m+1$ possible membership patterns for each $i$, and hence at most $(m+1)^N$ possible tuples $(S_{\rho_1}(u),\ldots,S_{\rho_m}(u))$. By \eqref{eq:Pu}, $\mathsf P_u(s)$ is completely determined by this tuple. As a consequence, \eqref{eq:diagonal-convex-combination} is a finite convex combination of commuting projections.

\section{Estimates for the lifted Walsh semigroup}
\label{sec:lower-estimate}

The representation \eqref{eq:diagonal-convex-combination} gives two estimates for the lifted semigroup $\mathsf S$ which will be used separately below. The lower estimate rests on the following lemma of Pisier \cite{Pisier}; we use the formulation in \cite[Lemma~7.4.25]{HNVW2}.

\begin{lemma}[Pisier]
\label{lem:lower}
Let $X$ be a Banach space of type $r\in(1,2]$, let $r'$ be the conjugate exponent, and let $\tau_{r,X}$ be the type-$r$ constant of $X$. If $P_1,\ldots,P_M\in\calL(X)$ are commuting contractive projections and
$$
 S=\sum_{m=1}^M\lambda_mP_m
$$
is a convex combination, then
$$
 \|Sx+x\|\geq \frac14(\pi\tau_{r,X})^{-r'}\|x\|.
$$
\end{lemma}

As before we fix $m\ge 1$ and let $Y_m^N=\Rad_m(L^p(\D^N;X))$, where $X$ is a UMD space and $1<p<\infty$. Fixing times $0<t_1\le \dots \le t_m$ we define the diagonal semigroup $\mathsf S$ on $Y_m^N$ as before.

\begin{proposition}[Uniform estimates for the lifted semigroup]
\label{prop:lifted-Walsh-estimates}
Under these assumptions there are constants $c_{p,X}>0$ and $K_{p,X}\ge 0$ such that, for all $N\ge 1$, $m\geq1$, and all $s>0$, we have
\begin{align}\label{eq:diagonal-lower}
 \|(I+\mathsf S(s))y\| \geq c_{p,X}\|y\|
 \end{align}
 and
 \begin{align}
 \label{eq:diagonal-real-bound}
 \|\mathsf S(s)\|_{\calL(Y_m^N)} \leq K_{p,X}.
\end{align}
The constants $c_{p,X}$ and $K_{p,X}$ are independent of $N,m,s$, and of the chosen times $t_1,\dots,t_m$.
\end{proposition}

\begin{proof}
Fix $s>0$ and put
$$
 \mathcal P_s:=\{\mathsf P_u(s):\,u\in[0,1]^N\}.
$$
This is a finite commuting family of projections. Putting $K_0:=\max\{1,K_{p,X}\}$, by \eqref{eq:product-Pu} we obtain
$$
 \sup\bigl\{\|Q_1\cdots Q_q\|:\,  q\geq1,\ Q_1,\ldots,Q_q\in\mathcal P_s \bigr\} \leq K_0,
$$
uniformly in all parameters. With the empty product interpreted as $I$, define
$$
 \|y\|_{s,\#} := \sup\bigl\{ \|Q_1\cdots Q_qy\|:\,q\geq0,\ Q_1,\ldots,Q_q\in\mathcal P_s \bigr\}.
$$
Then
$$
 \|y\|\leq\|y\|_{s,\#}\leq K_0\|y\|,
$$
and every operator in $\mathcal P_s$ is contractive for this norm.

Choose $r\in(1,2]$ so that all the spaces $Y_m^N$ have type $r$ with type constant at most $\tau$, where $r$ and $\tau$ depend only on $p$ and $X$; here we use that UMD spaces are $K$-convex and hence have non-trivial type (see Section~\ref{sec:preliminaries}). In the norm $\|\cdot\|_{s,\#}$ the type constant is at most $K_0\tau$, since
$$
 \Bigl(\E_\varepsilon \Bigl\|\sum_i\varepsilon_i y_i\Bigr\|_{s,\#}^2 \Bigr)^{1/2}
 \leq K_0\tau  \Bigl(\sum_i\|y_i\|_{s,\#}^r\Bigr)^{1/r}.
$$
By \eqref{eq:diagonal-convex-combination}, $\mathsf S(s)$ is a finite convex combination of members of $\mathcal P_s$. Pisier's lemma, applied in the equivalent norm, gives
$$
 \|(I+\mathsf S(s))y\|_{s,\#} \geq \frac14(\pi K_0\tau)^{-r'}\|y\|_{s,\#}.
$$
Returning to the original norm proves \eqref{eq:diagonal-lower}, with
$$
 c_{p,X}:=\frac1{4K_0}(\pi K_0\tau)^{-r'}.
$$
The bound \eqref{eq:diagonal-real-bound} follows directly from \eqref{eq:diagonal-convex-combination} and \eqref{eq:Stein-Pu},
$$
 \|\mathsf S(s)y\| \leq \E_u\|\mathsf P_u(s)y\|  \leq K_{p,X}\|y\|.
$$
\end{proof}

\begin{theorem}[Uniform real-time Walsh $R$-bound]
\label{thm:Walsh-real-Rbound}
Let $X$ be UMD and let $1<p<\infty$. Then
\begin{equation}
\label{eq:Walsh-real-Rbound}
 \sup_{N\geq1} \cR_{L^p(\D^N;X)} \bigl(\{T_N(t):\,t\geq0\}\bigr) \leq C_{p,X}.
\end{equation}
\end{theorem}

\begin{proof}
Take finitely many positive times $t_j>0$ and functions $f_j\in L^p(\D^N;X)$. After a simultaneous reordering of the pairs $(t_j,f_j)$, Proposition~\ref{prop:lifted-Walsh-estimates} with $s=1$ gives
$$
 \Bigl\|\sum_{j=1}^m\varepsilon_jT_N(t_j)f_j\Bigr\|_{\Rad_m(L^p(\D^N;X))}
 \leq K_{p,X} \Bigl\|\sum_{j=1}^m\varepsilon_jf_j\Bigr\|_{\Rad_m(L^p(\D^N;X))}.
$$
This proves the $R$-boundedness estimate for positive times. Since $T_N(t)\to I$ strongly as $t\downarrow0$ and strong operator closure does not increase an $R$-bound, it also includes $T_N(0)=I$.
\end{proof}

\section{The Walsh derivative \texorpdfstring{$R$}{R}-bound}
\label{sec:Walsh-proof}

We continue with a Walsh derivative $R$-bound.

\begin{theorem}\label{thm:Walsh-real-derivative-Rbound} Let $X$ be a UMD space and let $1<p<\infty$. For every $N\geq 1$ the family $\{t\Delta_NT_N(t):\,t>0\}$ is $R$-bounded, and
$$
 \sup_{N\geq 1} \cR_{L^p(\D^N;X)} \bigl(\{t\Delta_NT_N(t):\,t>0\}\bigr) \leq C_{p,X},
$$
where $C_{p,X}\geq 0$ is a constant depending only on $p$ and $X$.
\end{theorem}

\begin{proof}
Fix $N\ge 1$, $m\ge 1$, and times $t_1,\ldots,t_m>0$, ordered as in Section~\ref{sec:coupling}. On the Rademacher space $Y_m^N$, the diagonal semigroup $\mathsf S(s)=e^{-s\mathsf A}$ satisfies \eqref{eq:diagonal-lower} and \eqref{eq:diagonal-real-bound}. Moreover, \eqref{eq:diagonal-generator} shows that $\mathsf A$ is diagonal on the finite Walsh decomposition, with eigenvalues $t_j|A|\geq0$; in particular, $i\xi-\mathsf A$ is invertible for $\xi\ne 0$.

The next argument follows Pisier \cite{Pisier}; see also \cite[proof of Proposition~7.4.27]{HNVW2}.
For $\xi\in\mathbb R\setminus\{0\}$ and $s_0=\pi/|\xi|$, the integrated semigroup identity
$$
 -(I+\mathsf S(s_0))y = \int_0^{s_0} e^{i\xi s}\mathsf S(s)(i\xi-\mathsf A)y\dd s,
 \qquad y\in\Dom(\mathsf A),
$$
combined with \eqref{eq:diagonal-lower} and \eqref{eq:diagonal-real-bound}, gives
\begin{align*}
 \|(i\xi-\mathsf A)^{-1}\|_{\calL(Y^N_m)} \leq \frac{M}{|\xi|},
 \qquad M:=\frac{\pi K_{p,X}}{c_{p,X}}.
\end{align*}
Applying \cite[Lemma~10.33]{vanNeervenFA} to the operators $\pm i\mathsf A$ extends the estimate on the imaginary estimate to sectors about the two imaginary half-axes. Together with the Laplace resolvent estimate in the open left half-plane, which follows from \eqref{eq:diagonal-real-bound}, this gives an angle
$\rho\in(0,\pi/2)$ such that $\mathbb C\setminus\overline{\Sigma_\rho}
 \subseteq\varrho(\mathsf A)$ and
$$
 \sup_{\lambda\notin\overline{\Sigma_\rho}} \|\lambda(\lambda-\mathsf A)^{-1}\| \leq C_{p,X},
$$
uniformly in $N,m$ and in the chosen times. By the sectorial characterisation of bounded analytic semigroups (see, e.g., \cite[Theorem~G.5.2]{HNVW2} or \cite[Theorem 13.30]{vanNeervenFA}), this implies that the semigroups $\mathsf S$ extend boundedly to a common non-zero sector, uniformly with respect to the same parameters. Choose $r_{p,X}>0$ so small that the circle $|z-1|=r_{p,X}$ is contained in this sector. By Cauchy's formula,
$$
 \mathsf S'(1) = \frac{1}{2\pi i} \int_{|z-1|=r_{p,X}} \frac{\mathsf S(z)}{(z-1)^2}\dd z.
$$
Since $\mathsf S'(1)=-\mathsf A\mathsf S(1)$, estimating the contour integral gives
\begin{equation}
\label{eq:diagonal-derivative}
 \|\mathsf A\mathsf S(1)\|_{\calL(Y^N_m)} = \|\mathsf S'(1)\| \leq \frac{C_{p,X}}{r_{p,X}}.
\end{equation}
By the definitions of $\mathsf A$ and $\mathsf S$,
\begin{align*}
 \mathsf A\mathsf S(1) \Bigl(\sum_{j=1}^m\varepsilon_jf_j\Bigr)
 & = \mathsf A \Bigl( \sum_{j=1}^m\varepsilon_jT_N(t_j)f_j \Bigr) = \sum_{j=1}^m \varepsilon_jt_j\Delta_NT_N(t_j)f_j.
\end{align*}
Therefore \eqref{eq:diagonal-derivative} says that
\begin{align*}
 & \Bigl(
  \E_\varepsilon\Bigl\|\sum_{j=1}^m \varepsilon_jt_j\Delta_NT_N(t_j)f_j \Bigr\|_{L^p(\D^N;X)}^2 \Bigr)^{1/2}
 \\ &\qquad \qquad
 \leq C_{p,X}\Bigl(\E_\varepsilon\Bigl\| \sum_{j=1}^m\varepsilon_jf_j\Bigr\|_{L^p(\D^N;X)}^2 \Bigr)^{1/2}.
\end{align*}
Since the finite family was arbitrary, \eqref{eq:diagonal-derivative} gives the asserted uniform $R$-bound.
\end{proof}

\section{Walsh \texorpdfstring{$R$}{R}-analyticity}\label{sec:Walsh-R-analyticity}

We can now prove dimension-free $R$-analyticity for the Walsh semigroups.

\begin{theorem}[Dimension-free $R$-analyticity and $R$-sectoriality]
\label{thm:Walsh-R-analyticity} Let $X$ be a UMD space and let $1<p<\infty$. There exist an angle $\theta_{p,X}\in(0,\pi/2)$ and a constant $C_{p,X}\ge 0$ such that, for every $N\geq 1$, the Walsh semigroup $T_N$ is bounded analytic on $\Sigma_{\theta_{p,X}}$ and
$$
 \cR_{L^p(\D^N;X)} \bigl(\{T_N(z):\,z\in\Sigma_{\theta_{p,X}}\}\bigr) \leq C_{p,X}.
$$
In particular, the operators $\Delta_N$ are uniformly $R$-sectorial with a common angle smaller than ${\pi}/{2}$.
\end{theorem}

\begin{proof}
Let $R_0$ and $R_1$ denote the uniform $R$-bounds supplied by Theorems~\ref{thm:Walsh-real-Rbound} and \ref{thm:Walsh-real-derivative-Rbound}, respectively. For $n\geq 1$,
$$
 t^n\Delta_N^nT_N(t) =n^n\bigl((t/n)\Delta_NT_N(t/n)\bigr)^n.
$$
The product stability of $R$-bounded families \cite[Proposition~8.1.19]{HNVW2} therefore gives
$$
 \sup_{N\geq 1} \cR_{L^p(\D^N;X)} \bigl(\{t^n\Delta_N^nT_N(t):\,t>0\}\bigr) \leq n^nR_1^n.
$$
Choose $\delta\in(0,1)$ such that $e\delta R_1<1$. For $|z-t|<\delta t$, the $n$th terms in the Taylor expansions about $t>0$ then have $R$-bound at most
$$
 \frac{\delta^nn^n}{n!}R_1^n \leq (e\delta R_1)^n, \qquad n\geq 1,
$$
while the zeroth terms have $R$-bound at most $R_0$. Hence, by \cite[Proposition~8.1.24]{HNVW2}, these Taylor series form an $R$-bounded family, uniformly in $N$. Exactly as in the proof of the real characterisation of analytic semigroups (see, e.g., \cite[Theorem~13.31]{vanNeervenFA}), they agree on overlaps and glue to bounded analytic semigroups on the union of the discs $B(t;\delta t)$, $t>0$, and the resulting analytic semigroups agree with $e^{-t\Delta_N}$ on the positive real axis. This union contains every sector $\Sigma_\theta$ with $\sin\theta<\delta$; the same cited argument gives strong continuity at the vertex. The final assertion follows from the standard equivalence between $R$-analyticity and strict $R$-sectoriality recalled in Section~\ref{sec:preliminaries}.
\end{proof}

\section{From Walsh to Ornstein--Uhlenbeck}
\label{sec:CLT}

Let $X$ be a Banach space, let $1\leq p<\infty$, and let $\gamma_d$ denote the standard Gaussian measure on $\mathbb R^d$. Since the scalar Ornstein--Uhlenbeck operators $P_d(t)$ are positive contractions on $L^p(\mathbb R^d,\gamma_d)$, their tensor extensions $P_d(t)\otimes I_X$ define contractions on $L^p(\mathbb R^d,\gamma_d;X)$ and form a $C_0$-contraction semigroup on this space. We denote these extensions by $P_d(t)$ again, denote their generator by $L_d$, and put
$$
 A_d:=-L_d.
$$

The passage from Walsh to Ornstein--Uhlenbeck used here goes back to Beckner \cite{Beckner1975}; an operator-theoretic version was developed by Gon\c{c}alves \cite{Goncalves2016}. The following lemma records the $R$-bounded form needed for radial multipliers.

We use the probabilists' Hermite polynomials $H_k$, $k\in \mathbb N = \{0,1,2,\dots\}$, normalised so that
\begin{equation}
\label{eq:Hermite-generating-function}
 e^{zx-z^2/2} = \sum_{k=0}^\infty H_k(x)\frac{z^k}{k!}.
\end{equation}

\begin{lemma}[Transfer of $R$-multipliers]
\label{lem:chaos-multiplier-transfer}
Let $X$ be a Banach space, let $1\leq p<\infty$, and let $\Theta$ be an index set. For $\theta\in\Theta$, let $(m_\theta(k))_{k\geq 0}$ be a scalar sequence. For $N\geq 1$ define the Walsh multiplier
$$
 M_{\theta,N}^{\rm W} \Bigl(\sum_{A\subseteq\{1,\ldots,N\}}w_Ax_A\Bigr)
 := \sum_{A\subseteq\{1,\ldots,N\}} m_\theta(|A|)w_Ax_A.
$$
Suppose that for each $N\ge 1$ the operator family $\{M_{\theta,N}^{\rm W}:\,\theta\in\Theta\}$ is $R$-bounded on $L^p(\D^N;X)$, with uniform $R$-bound
$$
\sup_{N\geq 1} \cR_{L^p(\D^N;X)}\bigl(\{M_{\theta,N}^{\rm W}:\,\theta\in\Theta\}\bigr)\leq C.
$$
Then for every $d\geq 1$ and every finite subset $F\subseteq\mathbb N^d$ the Gaussian multiplier
$$
 M_{\theta,d}^{\rm G} \Bigl( \sum_{\alpha\in F}H_\alpha x_\alpha \Bigr) := \sum_{\alpha\in F} m_\theta(|\alpha|)H_\alpha x_\alpha,
$$
extends boundedly to $L^p(\gamma_d;X)$, the family $\{M_{\theta,d}^{\rm G}:\,\theta\in\Theta\}$
is $R$-bounded on $L^p(\gamma_d;X)$, and we have the uniform $R$-bound
$$
 \sup_{d\geq 1} \cR_{L^p(\gamma_d;X)} \bigl(\{M_{\theta,d}^{\rm G}:\,\theta\in\Theta\}\bigr) \leq C.
$$
\end{lemma}

\begin{proof}
Let $\xi_1,\ldots,\xi_M$ be independent signs, taking the values $-1$ and $1$ with equal probability, and, for $0\leq k\leq M$, put
$$
 \Phi_{k,M} := \frac{k!}{M^{k/2}} \sum_{1\leq r_1<\cdots<r_k\leq M} \xi_{r_1}\cdots\xi_{r_k}
$$
with the convention $\Phi_{0,M}:=1$. We use $\xi_r$ for these signs, reserving $\varepsilon_j$ for the external Rademacher randomisation. The normalisation is chosen so that
\begin{equation}
\label{eq:Walsh-Hermite-generating-function}
 \prod_{r=1}^M\left(1+\frac{z\xi_r}{\sqrt M}\right) = \sum_{k=0}^M\Phi_{k,M}\frac{z^k}{k!}.
\end{equation}

Put $G_M := M^{-1/2}\sum_{r=1}^M\xi_r$. Uniformly for $z$ in compact subsets of $\mathbb C$, expansion of the logarithm in the left-hand side of \eqref{eq:Walsh-Hermite-generating-function} gives
$$
 \log\prod_{r=1}^M\left(1+\frac{z\xi_r}{\sqrt M}\right) = zG_M-\frac{z^2}{2}+\mathcal R_M(z),
$$
where, for every compact set $K\subseteq\mathbb C$,
$$
 \sup_{z\in K}|\mathcal R_M(z)| \lesssim_K M^{-1/2}
$$
for all sufficiently large $M$, uniformly in the signs $\xi_r$. The scalar central limit theorem, Cauchy's coefficient formula, and \eqref{eq:Hermite-generating-function} therefore give, for every fixed $k_0$, the convergence
$$
 (\Phi_{0,M},\ldots,\Phi_{k_0,M}) \to (H_0(g),\ldots,H_{k_0}(g))
$$
jointly in distribution as $M\to\infty$, where $g$ is a standard real Gaussian variable; compare \cite{Beckner1975,Goncalves2016}. Notice also that $\Phi_{k,M}$ is a homogeneous Walsh polynomial of degree $k$.

By the orthogonality of the Walsh monomials,
$$
 \|\Phi_{k,M}\|_{L^2}^2 = \frac{(k!)^2}{M^k}\binom Mk \leq k!.
$$
Bonami's hypercontractive inequality \cite{Bonami1970} therefore implies that, for every fixed $k$ and $q\geq 2$,
$$
 \sup_{M\geq k}\|\Phi_{k,M}\|_{L^q}<\infty.
$$

For $d\geq 1$, consider $d$ mutually independent blocks, each consisting of $M$ independent signs, and write $\Phi_{k,M}^{(i)}$ for the corresponding copy in the $i$-th block. Let $g=(g_1,\ldots,g_d)$ be a standard Gaussian vector with law $\gamma_d$. For $\alpha=(\alpha_1,\ldots,\alpha_d)\in\mathbb N^d$ put
$$
 \Phi_{\alpha,M} := \prod_{i=1}^d\Phi_{\alpha_i,M}^{(i)}.
$$
Then, for every finite $F\subseteq\mathbb N^d$,
$$
 (\Phi_{\alpha,M})_{\alpha\in F} \to (H_\alpha(g))_{\alpha\in F}
$$
jointly in distribution. Moreover, $\Phi_{\alpha,M}$ is homogeneous of Walsh degree $|\alpha|$, so the approximation preserves the Walsh degree. Independence of the blocks and the preceding one-dimensional bounds also give, for every fixed $q\geq 2$,
$$
 \sup_{M\geq \max_i\alpha_i} \|\Phi_{\alpha,M}\|_{L^q} <\infty \qquad(\alpha\in\mathbb N^d).
$$

Fix $\theta_1,\ldots,\theta_n\in\Theta$ and $X$-valued Hermite
polynomials
$$
 f_j=\sum_{\alpha\in F_j}H_\alpha x_{j,\alpha}, \qquad j=1,\ldots,n.
$$
For $M$ sufficiently large define their Walsh lifts on $\D^{dM}$ by
$$
 f_j^{(M)} := \sum_{\alpha\in F_j}\Phi_{\alpha,M}x_{j,\alpha}.
$$
Since each $\Phi_{\alpha,M}$ has Walsh degree $|\alpha|$,
$$
 M_{\theta_j,dM}^{\rm W}f_j^{(M)}
 = \sum_{\alpha\in F_j} m_{\theta_j}(|\alpha|) \Phi_{\alpha,M}x_{j,\alpha},
$$
which is the Walsh lift of $M_{\theta_j,d}^{\rm G}f_j$.

Let $(\delta_j)_{j=1}^n$ be an independent Rademacher sequence and fix for the moment a vector $\delta=(\delta_1,\ldots,\delta_n)$ with each $\delta_j$ unimodular. The above joint convergence gives
$$
 \sum_{j=1}^n\delta_jf_j^{(M)} \to \sum_{j=1}^n\delta_jf_j
$$
in distribution, and likewise
$$
 \sum_{j=1}^n \delta_jM_{\theta_j,dM}^{\rm W}f_j^{(M)}
 \to \sum_{j=1}^n \delta_jM_{\theta_j,d}^{\rm G}f_j
$$
in distribution.

All these random variables take values in the finite-dimensional space
$$
 E_0:=\operatorname{span}\{x_{j,\alpha}:\,1\leq j\leq n, \ \alpha\in F_j\}\subseteq X.
$$
The asserted vector-valued convergence in distribution follows by applying a continuous linear map to the jointly convergent scalar families. The uniform moment bounds above, the finiteness of the sets $F_j$, and the finiteness of the multiplier coefficients occurring here imply that the norms of both families have uniformly bounded $q$-th moments for every fixed $q\geq 2$. Choosing $q>\max\{p,2\}$ gives uniform integrability of their $p$-th powers. Hence
$$
 \Bigl\| \sum_{j=1}^n\delta_jf_j^{(M)} \Bigr\|_{L^p(\D^{dM};X)}
 \to \Bigl\| \sum_{j=1}^n\delta_jf_j \Bigr\|_{L^p(\gamma_d;X)}
$$
and
$$
 \Bigl\| \sum_{j=1}^n \delta_jM_{\theta_j,dM}^{\rm W}f_j^{(M)} \Bigr\|_{L^p(\D^{dM};X)}
 \to \Bigl\| \sum_{j=1}^n \delta_jM_{\theta_j,d}^{\rm G}f_j \Bigr\|_{L^p(\gamma_d;X)}.
$$

By the assumed Walsh $R$-bound,
\begin{align*}
 &\Bigl(  \E_\delta  \Bigl\|  \sum_{j=1}^n  \delta_jM_{\theta_j,dM}^{\rm W}f_j^{(M)} \Bigr\|_{L^p(\D^{dM};X)}^2\Bigr)^{1/2}
 \!\!\leq
 C \Bigl( \E_\delta \Bigl\| \sum_{j=1}^n\delta_jf_j^{(M)} \Bigr\|_{L^p(\D^{dM};X)}^2 \Bigr)^{1/2}\!.
\end{align*}
The passage to the external Rademacher expectation is also uniform. Indeed, for every $\delta\in\mathbb T^n$,
$$
 \Bigl\|\sum_{j=1}^n\delta_jf_j^{(M)}\Bigr\|_p
 \leq \sum_{j=1}^n\sum_{\alpha\in F_j} \|\Phi_{\alpha,M}\|_p\|x_{j,\alpha}\|,
$$
and the right-hand side is bounded independently of $M$ and $\delta$; the same estimate applies to the transformed sums, with the finitely many additional factors $|m_{\theta_j}(|\alpha|)|$. Thus the squares of the two $L^p$-norms are dominated by constants independent of $M$ and $\delta$. Dominated convergence with respect to Haar measure on $\mathbb T^n$ therefore allows us to pass to the limit under $\E_\delta$. We obtain
$$
 \Bigl( \E_\delta \Bigl\| \sum_{j=1}^n \delta_jM_{\theta_j,d}^{\rm G}f_j \Bigr\|_{L^p(\gamma_d;X)}^2 \Bigr)^{1/2}
 \leq C \Bigl( \E_\delta \Bigl\| \sum_{j=1}^n\delta_jf_j \Bigr\|_{L^p(\gamma_d;X)}^2 \Bigr)^{1/2}.
$$
Taking $n=1$ first shows that every $M_{\theta,d}^{\rm G}$ extends boundedly to $L^p(\gamma_d;X)$ with norm at most $C$. Since the $X$-valued Hermite polynomials are dense in this space, approximation of a general finite family $f_1,\ldots,f_n$ completes the proof.
\end{proof}

Assume from now on that $X$ is UMD and $1<p<\infty$, and write $P_d(t)=e^{-tA_d}$ on $L^p(\gamma_d;X)$. For the standard multivariate Hermite polynomials one has
$$
 A_dH_\alpha=|\alpha|H_\alpha, \qquad P_d(t)H_\alpha=e^{-t|\alpha|}H_\alpha.
$$
On $X$-valued Hermite polynomials, the symbols $e^{-tk}$ and $tke^{-tk}$ therefore give $P_d(t)$ and $tA_dP_d(t)$, respectively.

First take $\Theta=[0,\infty)$ and $m_t(k)=e^{-tk}$. By \eqref{eq:Walsh-real-Rbound} the corresponding Walsh semigroup multipliers are $R$-bounded uniformly in the dimension. Hence Lemma~\ref{lem:chaos-multiplier-transfer} gives bounded Gaussian multipliers $M_{t,d}^{\rm G}$ satisfying
$$
 M_{t,d}^{\rm G}H_\alpha=e^{-t|\alpha|}H_\alpha.
$$
Since the $X$-valued Hermite polynomials are dense and $P_d(t)$ has the same action on them, $M_{t,d}^{\rm G}=P_d(t)$, and
\begin{align}\label{eq:R-bound-Pd}
 \sup_{d\geq 1} \cR_{L^p(\gamma_d;X)} \bigl(\{P_d(t):\,t\geq 0\}\bigr) \leq C_{p,X}.
\end{align}

For the derivative family take $\Theta=(0,\infty)$ and $m_t(k)=tke^{-tk}$. Lemma~\ref{lem:chaos-multiplier-transfer} gives bounded operators $M_{t,d}^{\rm G}$ satisfying, on $X$-valued Hermite polynomials,
$$
 M_{t,d}^{\rm G}f=tA_dP_d(t)f.
$$
Let $f_n$ be $X$-valued Hermite polynomials converging to $f\in L^p(\gamma_d;X)$. Then $P_d(t)f_n\to P_d(t)f$ and
$$
 A_dP_d(t)f_n =t^{-1}M_{t,d}^{\rm G}f_n \to t^{-1}M_{t,d}^{\rm G}f.
$$
Since $A_d$ is closed, it follows that $P_d(t)f\in\Dom(A_d)$ and $M_{t,d}^{\rm G}f=tA_dP_d(t)f$. Thus
\begin{align}\label{eq:derivative-R-bound-Pd}
 \sup_{d\geq 1} \cR_{L^p(\gamma_d;X)} \bigl(\{tA_dP_d(t):\,t>0\}\bigr) \leq C_{p,X}.
\end{align}

\begin{theorem}[Dimension-free $R$-sectoriality of the Ornstein--Uhlenbeck operator]
\label{thm:main-Rsectorial}
Let $X$ be a UMD space and let $1<p<\infty$. There exist angles $\theta_{p,X},\rho_{p,X}\in (0,\pi/2)$ and a constant $C_{p,X}\geq 0$ such that, for every $d\geq 1$, the semigroup $P_d$ extends analytically to $\Sigma_{\theta_{p,X}}$ and
\begin{equation}
\label{eq:Gaussian-complex-time-Rbound}
 \sup_{d\geq 1} \cR_{L^p(\gamma_d;X)} \bigl(\{P_d(z):\,z\in\Sigma_{\theta_{p,X}}\}\bigr) \leq C_{p,X}.
\end{equation}
Moreover,
\begin{equation}
\label{eq:Gaussian-sectorial-Rbound}
 \sup_{d\geq 1} \cR_{L^p(\mathbb R^d,\gamma_d;X)}
 \Bigl( \bigl\{
   \lambda(\lambda-A_d)^{-1}:\, \lambda\notin\overline{\Sigma_{\rho_{p,X}}}
  \bigr\} \Bigr)
 \leq C_{p,X}.
\end{equation}
In particular, the operators $A_d$ are uniformly $R$-sectorial with a common angle strictly smaller than $\pi/2$.
\end{theorem}

\begin{proof}
For $s>0$, the preceding identification of $A_dP_d(s)$ shows that $P_d(s)$ maps the whole space into $\Dom(A_d)$. The standard semigroup commutation relation
$$
 A_dP_d(s)x=P_d(s)A_dx, \qquad x\in\Dom(A_d),
$$
then gives inductively
$$
 P_d(ns)L^p(\gamma_d;X)\subseteq\Dom(A_d^n),
 \qquad A_d^nP_d(ns)=\bigl(A_dP_d(s)\bigr)^n.
$$
Putting $s=t/n$, we obtain
$$
 t^nA_d^nP_d(t) = n^n\bigl((t/n)A_dP_d(t/n)\bigr)^n.
$$
The product stability of $R$-bounded families gives
$$
 \sup_{d\geq 1} \cR_{L^p(\mathbb R^d,\gamma_d;X)} \bigl(\{t^nA_d^nP_d(t):\,t>0\}\bigr) \leq n^n C_{p,X}^n.
$$
After increasing $C_{p,X}$ if necessary, we may assume that it also dominates the $R$-bound in \eqref{eq:R-bound-Pd}. Choose $\delta\in(0,1)$ so small that $e\delta C_{p,X}<1$. The Taylor-series argument in the proof of Theorem~\ref{thm:Walsh-R-analyticity}, applied to the preceding higher-derivative bounds and to \eqref{eq:R-bound-Pd}, gives \eqref{eq:Gaussian-complex-time-Rbound}. The equivalence between $R$-analyticity and strict $R$-sectoriality then gives \eqref{eq:Gaussian-sectorial-Rbound} for some $\rho_{p,X}<\pi/2$.
\end{proof}

\section{Proof of Theorem~\ref{thm:main-Hinfty}}
\label{sec:Hinfty-consequence}

Put
$$
 E_{\mathbb N} := L^p(\mathbb R^{\mathbb N},\gamma_{\mathbb N};X),
$$
where $\gamma_{\mathbb N}$ is the countable product of one-dimensional standard Gaussian measures. On cylindrical functions (i.e., functions effectively depending only on finitely many coordinates), define $P_{\mathbb N}(t)$ by applying the corresponding finite-dimensional Ornstein--Uhlenbeck semigroup. These compatible contractions extend, by density, to a $C_0$-semigroup on $E_{\mathbb N}$; denote its positive generator by $A_{\mathbb N}$.

On the underlying scalar $L^p$-space, $P_{\mathbb N}$ is a positive $C_0$-contraction semigroup, and its action on $E_{\mathbb N}$ is the tensor extension by $I_X$. The Hieber--Pr\"uss theorem \cite{HieberPruss} (see also \cite[Theorem~10.7.12]{HNVW2}) therefore gives, for every $\nu\in(\pi/2,\pi)$, a constant $H_{p,X,\nu}\ge 0$ such that
$$
 \|f(A_{\mathbb N})\| \leq H_{p,X,\nu}\|f\|_{H^\infty(\Sigma_\nu)},
 \qquad f\in H^1(\Sigma_\nu)\cap H^\infty(\Sigma_\nu).
$$
This means that $A_{\mathbb N}$ has a bounded $H^\infty(\Sigma_\nu)$-calculus.

The dimension-free $R$-analyticity proved in Theorem~\ref{thm:main-Rsectorial} passes to the infinite product. Choose a common sector $\Sigma_\theta$ on which the finite-dimensional analytic extensions are uniformly $R$-bounded. These extensions are compatible: if $d'\geq d$ and a function depends only on the first $d$ coordinates, the $d'$- and $d$-dimensional operators agree for positive real times and hence, by uniqueness of analytic continuation, throughout $\Sigma_\theta$.

If $z_1,\ldots,z_m\in\Sigma_\theta$ and $f_1,\ldots,f_m\in E_{\mathbb N}$ are cylindrical, all the $f_j$ depend on the first $d$ coordinates for some $d$. The finite-dimensional $R$-bound gives
$$
 \Bigl\| \sum_{j=1}^m\varepsilon_jP_{\mathbb N}(z_j)f_j \Bigr\|_{\Rad_m(E_{\mathbb N})}
 \leq C_{p,X} \Bigl\| \sum_{j=1}^m\varepsilon_jf_j \Bigr\|_{\Rad_m(E_{\mathbb N})}.
$$
By density of cylindrical functions, this estimate extends to arbitrary $f_1,\ldots,f_m\in E_{\mathbb N}$. The same compatibility and the uniform bounds on compact subsectors extend the finite-dimensional operators to a bounded analytic semigroup on $E_{\mathbb N}$ whose restriction to the positive axis is $P_{\mathbb N}$. Hence $\omega_R(A_{\mathbb N})<\pi/2$.

Since $X$ is UMD, so is $E_{\mathbb N}$, and UMD spaces have the triangular contraction property by \cite[Theorem~7.5.9]{HNVW2}. The Kalton--Weis comparison theorem in the form of \cite[Corollary~10.4.10]{HNVW2} therefore gives
$$
 \omega_{H^\infty}(A_{\mathbb N}) = \omega_R(A_{\mathbb N}) < \frac{\pi}{2}.
$$
Choose
$$
 \omega_{H^\infty}(A_{\mathbb N}) < \sigma_{p,X}<\frac{\pi}{2}.
$$
Then $A_{\mathbb N}$ has a bounded $H^\infty(\Sigma_{\sigma_{p,X}})$-functional calculus; denote its bound by $C_{p,X}$.

For $d\geq 1$, let
$$
 J_d:L^p(\mathbb R^d,\gamma_d;X)\to E_{\mathbb N}
$$
be the isometric embedding as functions of the first $d$ coordinates, and let $Q_d$ be conditional expectation onto these coordinates, followed by the natural identification with $L^p(\mathbb R^d,\gamma_d;X)$. Then $Q_dJ_d=I$ and
$$
 P_{\mathbb N}(t)J_d=J_dP_d(t), \qquad Q_dP_{\mathbb N}(t)=P_d(t)Q_d,
 \qquad t\geq 0.
$$
For $a>0$, the Laplace formula gives
$$
 (-a-A_{\mathbb N})^{-1}J_d = J_d(-a-A_d)^{-1},
 \qquad
 Q_d(-a-A_{\mathbb N})^{-1} = (-a-A_d)^{-1}Q_d.
$$
By the resolvent identities and analytic continuation, these relations extend to the contours used in the Dunford calculus, and therefore
$$
 f(A_d)=Q_df(A_{\mathbb N})J_d.
$$
It follows that
$$
 \|f(A_d)\| \leq C_{p,X}\|f\|_{H^\infty(\Sigma_{\sigma_{p,X}})},
 \qquad d\geq 1,
$$
for all $f\in H^1(\Sigma_{\sigma_{p,X}})\cap H^\infty(\Sigma_{\sigma_{p,X}})$. This proves Theorem~\ref{thm:main-Hinfty}.

\section{Necessity of the UMD condition}
\label{sec:OU-Hinfty-necessity}

In this final section we give a short proof of the result, due to Betancor, Castro, Curbelo and Rodr\'iguez-Mesa \cite[Proposition~2.2]{BetancorEtAl}, that the UMD property of $X$ is necessary for the existence of a bounded $H^\infty(\Sigma_\sigma)$-calculus on $L^p(\R^d,\gamma_d;X)$ for a single angle $0<\sigma<\pi$, dimension $d\geq 1$, and exponent $1<p<\infty$. Its main idea is that, after a ground state transform and dilation, the Ornstein--Uhlenbeck operator converges at small spatial scales to the Euclidean Laplacian. This reduces the assertion to the classical theorem of Guerre-Delabri\`ere that bounded imaginary powers of the vector-valued Euclidean Laplacian force the UMD property. A detailed modern proof of Guerre-Delabri\`ere's theorem, including a quantitative estimate for the UMD constant, is given in \cite[Corollary~10.5.2]{HNVW2}. The notes to \cite[Section~10.5]{HNVW2} explain that the quantitative estimate, although not stated explicitly in \cite{GuerreDelabriere1991}, is implicit in the original argument.

Fix $1<p<\infty$, let $X$ be a complex Banach space, and let
$$
 A_d=-\Delta+x\cdot\nabla
$$
denote the positive Ornstein--Uhlenbeck operator on $L^p(\R^d,\gamma_d;X)$, with semigroup $P_d(t)=e^{-tA_d}$.

\begin{theorem}[A bounded $H^\infty$-calculus forces UMD]
\label{thm:OU-Hinfty-implies-UMD}
Let $X$ be a Banach space, let $1<p<\infty$, and suppose that $A_d$ has a bounded $H^\infty(\Sigma_\sigma)$-calculus on $L^p(\R^d,\gamma_d;X)$ for some $d\ge 1$ and angle $0<\sigma<\pi$, with bound $C_{p,d,\sigma}$. Then $X$ is a UMD space, and we have
\begin{equation}
\label{eq:beta-vs-calculus}
 \beta_{p,X}\le C_{p,d,\sigma}.
\end{equation}
\end{theorem}

\begin{proof}
We first reduce to dimension $d=1$. Let
$$
 J:L^p(\R,\gamma_1;X)\to L^p(\R^d,\gamma_d;X)
$$
be the isometric embedding as functions of the first coordinate, and let $Q$ be conditional expectation onto that coordinate, followed by the natural identification with $L^p(\R,\gamma_1;X)$. Then $QJ=I$ and
$$
 P_d(t)J=JP_1(t), \qquad QP_d(t)=P_1(t)Q, \qquad t\geq 0.
$$
These identities pass to the resolvents and the Dunford calculi, so
$$
 f(A_1)=Qf(A_d)J.
$$
Thus $A_1$ has a bounded $H^\infty(\Sigma_\sigma)$-calculus on $L^p(\R,\gamma_1;X)$ with bound at most $C_{p,d,\sigma}$.

In order to pass from Gaussian to Lebesgue measure we define
$$
 (U_pf)(x):=(2\pi)^{-1/(2p)}e^{-x^2/(2p)}f(x),
$$
which is an isometry from $L^p(\R,\gamma_1;X)$ onto $L^p(\R;X)$, and put
$$
 B:=U_pA_1U_p^{-1}.
$$
A direct calculation on $C_{\rm c}^\infty(\R)\otimes X$ gives
\begin{equation}
\label{eq:ground-state-B}
 B
 = -\frac{\!\dd^2}{\!\dd x^2}
   +\Bigl(1-\frac2p\Bigr)x\frac{\!\dd}{\!\dd x}
   +\frac{p-1}{p^2}x^2-\frac1p.
\end{equation}

For $R\geq 1$ let
$$
 (D_Rg)(x):=R^{1/p}g(Rx)
$$
be the isometric dilation on $L^p(\R;X)$, and set
$$
 B_R:=R^{-2}D_R^{-1}BD_R.
$$
Again by direct computation,
\begin{equation}
\label{eq:blow-up-BR}
 B_R = -\frac{\!\dd^2}{\!\dd x^2}
   +\frac{1-2/p}{R^2}x\frac{\!\dd}{\!\dd x}
   +\frac{p-1}{p^2R^4}x^2-\frac1{pR^2}
\end{equation}
on $C_{\rm c}^\infty(\R)\otimes X$. Hence, with
$$
 L:=-\frac{\!\dd^2}{\!\dd x^2},
$$
we have
\begin{equation}
\label{eq:BR-core-convergence}
 B_Rg\to Lg \quad\hbox{in }L^p(\R;X)
\end{equation}
for all $g\in C_{\rm c}^\infty(\R)\otimes X$.

The operators $B_R$ are obtained from $A_1$ by an isometric similarity and a rescaling. They therefore have bounded $H^\infty(\Sigma_\sigma)$-calculi with bound at most $C_{p,d,\sigma}$. Moreover,
$$
 e^{-tB_R} = D_R^{-1}U_pP_1(t/R^2)U_p^{-1}D_R, \qquad t\geq 0,
$$
so $-B_R$ generates a $C_0$-contraction semigroup.

Fixing an arbitrary sequence $R_n\to\infty$, we will apply the Trotter--Kato approximation theorem in the form of \cite[Theorem~III.4.8]{EngelNagel} to the generators $G:=-L$ and $G_n:=-B_{R_n}$.

The $C_0$-semigroup generated by $G$ is the heat semigroup $(e^{-tL})_{t\ge0}$ on $L^p(\R;X)$. By the preceding representation of $e^{-tB_R}$, each operator $G_n$ likewise generates a $C_0$-contraction semigroup. Thus the stability hypothesis of \cite[Theorem~III.4.8]{EngelNagel} holds with $M=1$ and $\omega=0$.

We next claim that
$$
 \mathscr D:=C_{\rm c}^\infty(\R)\otimes X
$$
is a core for $G$. To see this, we note that if $g\in\mathsf D(L)$, then $g_\varepsilon=e^{-\varepsilon L}g$ converges to $g$ in the graph norm of $L$ as $\varepsilon\downarrow0$. The heat kernel representation shows that $g_\varepsilon$ and its first two derivatives belong to $L^p(\R;X)$. Multiplication by smooth cut-off functions therefore approximates $g_\varepsilon$ in the graph norm by compactly supported elements of $W^{2,p}(\R;X)$. Mollification gives compactly supported smooth $X$-valued functions, and approximation of the resulting Bochner convolution integrals by finite sums gives approximation by elements of $C_{\rm c}^\infty(\R)\otimes X$ in the same norm.

The calculation leading to \eqref{eq:blow-up-BR} shows that $\mathscr D\subseteq\mathsf D(B)$; since $\mathscr D$ is invariant under the dilations $D_R$, it follows from $B_R = R^{-2}D_R^{-1}BD_R$ that $\mathscr D\subseteq\mathsf D(B_R)$ for every $R\geq 1$. Now
\eqref{eq:BR-core-convergence} gives
$$
 G_ng = -B_{R_n}g \to -Lg = Gg, \qquad g\in\mathscr D.
$$
Hence condition~(a) of \cite[Theorem~III.4.8]{EngelNagel} is satisfied. It follows that
$$
 e^{-tB_{R_n}}g\to e^{-tL}g \quad\hbox{in }L^p(\R;X),
$$
for every $g\in L^p(\R;X)$, uniformly for $t$ in compact subsets of $[0,\infty)$. Since the sequence $(R_n)$ was arbitrary, this proves
\begin{equation}
\label{eq:semigroup-convergence}
 e^{-tB_R}g\to e^{-tL}g \quad\hbox{in }L^p(\R;X), \qquad t\geq 0,
\end{equation}
as $R\to\infty$.

Choose $\nu$ with $\omega(A_1)<\nu<\sigma$. Similarity and rescaling give a uniform sectorial resolvent bound for $B_R$ outside $\overline{\Sigma_\nu}$. For $a>0$, the Laplace formula and
\eqref{eq:semigroup-convergence} give
$$
 (-a-B_R)^{-1}g = -\int_0^\infty e^{-at}e^{-tB_R}g\dd t
 \to -\int_0^\infty e^{-at}e^{-tL}g\dd t = (-a-L)^{-1}g.
$$
Thus the resolvents converge strongly on the negative real axis.

Let us prove next that this convergence propagates to the whole connected region $\mathbb C\setminus\overline{\Sigma_\nu}$. Suppose it is known at a point $z_0$ in this region. The uniform sectorial estimates bound $(z_0-B_R)^{-1}$ uniformly in $R$, and for $z$ in a sufficiently small disc about $z_0$ the resolvent identity gives the uniformly convergent Neumann expansion
$$
 (z-B_R)^{-1} = (z_0-B_R)^{-1} \sum_{k=0}^\infty \bigl(-(z-z_0)(z_0-B_R)^{-1}\bigr)^k.
$$
The analogous expansion for the resolvent of $L$, together with strong convergence of every power of $(z_0-B_R)^{-1}$ and uniform convergence of the two Neumann series, identifies the limit with $(z-L)^{-1}g$. Thus the convergence is locally uniform throughout a smaller disc. A finite chain of overlapping discs along a path from the negative real axis to any prescribed point proves
\begin{equation}
\label{eq:BR-resolvent-convergence}
 (z-B_R)^{-1}g
 \to
 (z-L)^{-1}g,
 \qquad z\notin\overline{\Sigma_\nu}.
\end{equation}

Now choose $\eta$ with $\nu<\eta<\sigma$. On $\partial\Sigma_\eta$, the sectorial estimates give
$$
 \sup_{R\geq 1}\|(z-B_R)^{-1}\| \lesssim_\eta \frac1{|z|}.
$$
Combining this domination with \eqref{eq:BR-resolvent-convergence} and the definition of $H^1(\Sigma_\sigma)$, dominated convergence in the Dunford formula gives, for every $f\in H^1(\Sigma_\sigma)\cap H^\infty(\Sigma_\sigma)$,
$$
 f(B_R)g\to f(L)g.
$$
It follows that
\begin{equation}
\label{eq:L-Hinfty-bound}
 \|f(L)\|_{\mathscr L(L^p(\R;X))} \le C_{p,d,\sigma}\|f\|_{H^\infty(\Sigma_\sigma)},
 \qquad f\in H^1(\Sigma_\sigma) \cap H^\infty(\Sigma_\sigma).
\end{equation}
Hence the Euclidean Laplacian $L$ has a bounded $H^\infty(\Sigma_\sigma)$-calculus with bound at most $C_{p,d,\sigma}$.

We next pass from the bounded $H^\infty$-calculus of $L$ to its
imaginary powers.

The heat semigroup $e^{-tL}$ is strongly stable on $L^p(\R;X)$. Indeed, if
$h_t$ denotes the heat kernel and
$g\in L^1(\R;X)\cap L^p(\R;X)$, then Young's inequality gives
$$
 \|e^{-tL}g\|_p \leq \|h_t\|_p\|g\|_1 \to 0 \qquad(t\to\infty),
$$
and the general case follows by density and contractivity. It follows that the range of $L$ is dense, since
$$
 L\int_0^T e^{-tL}g\dd t = g-e^{-TL}g \to g \qquad(T\to\infty).
$$
The operator $L$ is densely defined, being the negative of the generator of a $C_0$-semigroup. Hence $L$ is standard sectorial by \cite[Proposition~15.3.2]{HNVW3}; in particular, it is injective.

Following \cite[Definition~15.2.2]{HNVW3}, for $s\in\R$ the imaginary power $L^{is}$ is defined through the extended Dunford calculus by
$$
 L^{is}=f_s(L),
 \qquad
 f_s(z):=z^{is}:=\exp(is\log z),
$$
where $\log$ denotes the branch which is holomorphic on $\mathbb C\setminus(-\infty,0]$. For $z\in\Sigma_\sigma$ one has
$$
 |f_s(z)|
 =
 e^{-s\arg z},
$$
and therefore
$$
 \|f_s\|_{H^\infty(\Sigma_\sigma)}
 =
 e^{\sigma|s|}.
$$
By \cite[Theorem~15.1.17]{HNVW3}, the operator $f_s(L)$ defined through the extended Dunford calculus is bounded and agrees with the operator supplied by the bounded $H^\infty(\Sigma_\sigma)$-calculus. The estimate \eqref{eq:L-Hinfty-bound} therefore gives
$$
 \|L^{is}\|_{\mathscr L(L^p(\R;X))} \leq C_{p,d,\sigma}e^{\sigma|s|}, \qquad s\in\R.
$$
Thus $L$ has bounded imaginary powers in the sense of
\cite[Definition~15.3.4]{HNVW3}. In particular,
$$
 \liminf_{s\downarrow0} \|L^{is}\|_{\mathscr L(L^p(\R;X))} \leq C_{p,d,\sigma}.
$$

By the Guerre-Delabri\`ere theorem in the formulation \cite[Corollary~10.5.2]{HNVW2}, $X$ is UMD and
$$
 \beta_{p,X} \le \liminf_{s\downarrow0}\|L^{is}\|_{\mathscr L(L^p(\R;X))} \le C_{p,d,\sigma}.
$$
This proves both assertions.
\end{proof}

\begin{remark}[Comparison with the direct Ornstein--Uhlenbeck criterion]
\label{rem:Betancor-comparison}
Betancor, Castro, Curbelo and Rodr\'iguez-Mesa proved a more specific Ornstein--Uhlenbeck characterisation. In the normalisation
$$
 \mathcal O=-\frac12\frac{\!\dd^2}{\!\dd y^2}+y\frac{\!\dd}{\!\dd y}
$$
on $L^p(\R,e^{-y^2}\!\dd y;X)$, they show that boundedness of the shifted imaginary powers
$$
 \Bigl(\mathcal O+\frac12\Bigr)^{i\gamma},
 \qquad \gamma\in\R,
$$
for some (equivalently, for every) $1<p<\infty$, is equivalent to the UMD property of $X$; see \cite[Proposition~2.2]{BetancorEtAl}. Their proof is based on a local comparison of the Ornstein--Uhlenbeck and Euclidean kernels. Theorem~\ref{thm:OU-Hinfty-implies-UMD} uses a different mechanism: the dilation in \eqref{eq:blow-up-BR} makes the Euclidean
Laplacian appear as a small-scale limit.
\end{remark}

\section*{AI disclosure statement}

This manuscript is the outgrowth of extensive conversations with OpenAI's ChatGPT 5.6 model. The elegant Rademacher lift technique was discovered in this process. ChatGPT was used as an aid in exploratory mathematical discussion, verification of arguments, and editorial revision. All statements and proofs in the final manuscript have been reviewed by the author, who takes full responsibility for their correctness. I thank Emiel Lorist and Mark Veraar for useful suggestions.


\begin{thebibliography}{99}

\bibitem{ArhancetNC}
C.~Arhancet,
Analytic semigroups on vector valued noncommutative $L^p$-spaces,
\emph{Studia Math.} \textbf{216} (2013), 271--290.
\href{https://doi.org/10.4064/sm216-3-5}{doi:10.4064/sm216-3-5}.

\bibitem{ArhancetPisier}
C.~Arhancet,
On a conjecture of Pisier on the analyticity of semigroups,
\emph{Semigroup Forum} \textbf{91} (2015), 450--462.
\href{https://doi.org/10.1007/s00233-015-9715-3}
{doi:10.1007/s00233-015-9715-3}.

\bibitem{ArhancetDilation}
C.~Arhancet,
Dilations of semigroups on von Neumann algebras and noncommutative
$L^p$-spaces,
\emph{J. Funct. Anal.} \textbf{276} (2019), 2279--2314.
\href{https://doi.org/10.1016/j.jfa.2018.11.013}
{doi:10.1016/j.jfa.2018.11.013}.

\bibitem{ArhancetBochner}
C.~Arhancet,
On analyticity of semigroups on Bochner spaces and on vector-valued
noncommutative $L^p$-spaces,
arXiv:1807.00875.
\href{https://doi.org/10.48550/arXiv.1807.00875}
{doi:10.48550/arXiv.1807.00875}.

\bibitem{AFLM}
C.~Arhancet, S.~Fackler and C.~Le Merdy,
Isometric dilations and $H^\infty$ calculus for bounded analytic
semigroups and Ritt operators,
\emph{Trans. Amer. Math. Soc.} \textbf{369} (2017), 6899--6933.
\href{https://doi.org/10.1090/tran/6849}{doi:10.1090/tran/6849}.

\bibitem{Beckner1975}
W.~Beckner,
\emph{Inequalities in Fourier analysis},
Ann. of Math. (2) \textbf{102} (1975), no.~1, 159--182.
\href{https://doi.org/10.2307/1970980}{doi:10.2307/1970980}.

\bibitem{BetancorEtAl}
J.~J. Betancor, A.~J. Castro, J.~Curbelo and L.~Rodr\'iguez-Mesa,
\emph{Characterization of UMD Banach spaces by imaginary powers of
Hermite and Laguerre operators},
Complex Anal. Oper. Theory \textbf{7} (2013), no.~4, 1019--1048.
\href{https://doi.org/10.1007/s11785-011-0203-9}
{doi:10.1007/s11785-011-0203-9}.

\bibitem{Bonami1970}
A.~Bonami,
\emph{\'Etude des coefficients de Fourier des fonctions de $L^p(G)$},
Ann. Inst. Fourier (Grenoble) \textbf{20} (1970), no.~2, 335--402.
\href{https://doi.org/10.5802/aif.357}{doi:10.5802/aif.357}.

\bibitem{CarbonaroDragicevic}
A.~Carbonaro and O.~Dragi\v{c}evi\'{c},
Functional calculus for generators of symmetric contraction semigroups,
\emph{Duke Math. J.} \textbf{166} (2017), 937--974.
\href{https://doi.org/10.1215/00127094-3774526}
{doi:10.1215/00127094-3774526}.

\bibitem{EngelNagel}
K.-J.~Engel and R.~Nagel,
\emph{One-Parameter Semigroups for Linear Evolution Equations},
Graduate Texts in Mathematics, vol.~194,
Springer-Verlag, New York, 2000.
\href{https://doi.org/10.1007/b97696}{doi:10.1007/b97696}.

\bibitem{Goncalves2016}
F.~Gon\c{c}alves,
\emph{A central limit theorem for operators},
J. Funct. Anal. \textbf{271} (2016), no.~6, 1585--1603.
\href{https://doi.org/10.1016/j.jfa.2016.06.008}
{doi:10.1016/j.jfa.2016.06.008}.

\bibitem{GuerreDelabriere1991}
S.~Guerre-Delabri\`ere,
\emph{Some remarks on complex powers of $(-\Delta)$ and UMD spaces},
Illinois J. Math. \textbf{35} (1991), no.~3, 401--407.
\href{https://doi.org/10.1215/ijm/1255987786}
{doi:10.1215/ijm/1255987786}.

\bibitem{HieberPruss}
M.~Hieber and J.~Pr\"uss,
Functional calculi for linear operators in vector-valued $L^p$-spaces via
the transference principle,
\emph{Adv. Differential Equations} \textbf{3} (1998), 847--876.
\href{https://doi.org/10.57262/ade/1366292551}
{doi:10.57262/ade/1366292551}.

\bibitem{HNVW1}
T.~Hyt\"onen, J.~van Neerven, M.~Veraar and L.~Weis,
\emph{Analysis in Banach Spaces, Volume I: Martingales and Littlewood--Paley Theory}, Springer, 2016.
\href{https://doi.org/10.1007/978-3-319-48520-1}
{doi:10.1007/978-3-319-48520-1}.

\bibitem{HNVW2}
T.~Hyt\"onen, J.~van Neerven, M.~Veraar and L.~Weis,
\emph{Analysis in Banach Spaces, Volume II: Probabilistic Methods and
Operator Theory}, Springer, 2017.
\href{https://doi.org/10.1007/978-3-319-69808-3}
{doi:10.1007/978-3-319-69808-3}.

\bibitem{HNVW3}
T.~Hyt\"onen, J.~van Neerven, M.~Veraar and L.~Weis,
\emph{Analysis in Banach Spaces, Volume III: Harmonic Analysis and Spectral Theory}, Springer, 2023.
\href{https://doi.org/10.1007/978-3-031-46598-7}
{doi:10.1007/978-3-031-46598-7}.

\bibitem{HytonenLPS}
T.~P. Hyt\"onen,
Littlewood--Paley--Stein theory for semigroups in UMD spaces,
\emph{Rev. Mat. Iberoam.} \textbf{23} (2007), 973--1009.
\href{https://doi.org/10.4171/RMI/521}{doi:10.4171/RMI/521}.

\bibitem{KaltonWeis}
N.~J. Kalton and L.~Weis,
The $H^\infty$-calculus and sums of closed operators,
\emph{Math. Ann.} \textbf{321} (2001), 319--345.
\href{https://doi.org/10.1007/s002080100231}
{doi:10.1007/s002080100231}.

\bibitem{vanNeervenFA}
J.~van Neerven,
\emph{Functional Analysis},
Cambridge Studies in Advanced Mathematics, vol.~201,
Cambridge University Press, Cambridge, 2022;
second edition forthcoming.
\href{https://doi.org/10.1017/9781009232487}
{doi:10.1017/9781009232487}.

\bibitem{Neerven-Stein}
J.M.A.M.~van Neerven,
\emph{Stein's inequality and dimension-free $R$-sectoriality of vector-valued Ornstein--Uhlenbeck operators}, in preparation.

\bibitem{Neerven-Rsectorial}
J.M.A.M.~van Neerven,
\emph{On the optimal angle of dimension-free analyticity for vector-valued Ornstein--Uhlenbeck semigroups}, in preparation.

\bibitem{Pisier}
G.~Pisier,
Holomorphic semigroups and the geometry of Banach spaces,
\emph{Ann. of Math.} \textbf{115} (1982), 375--392.
\href{https://doi.org/10.2307/1971396}{doi:10.2307/1971396}.


\bibitem{Xu}
Q.~Xu,
$H^\infty$ functional calculus and maximal inequalities for semigroups of
contractions on vector-valued $L_p$-spaces,
\emph{Int. Math. Res. Not. IMRN} (2015), no.~14, 5715--5732.
\href{https://doi.org/10.1093/imrn/rnu104}{doi:10.1093/imrn/rnu104}.

\end{thebibliography}
\end{document}